\documentclass[10pt]{amsart}

\usepackage{amsmath, amssymb}
\input xy
\xyoption{all}

\newtheorem{theorem}{Theorem}[section]

\newtheorem{corollary}[theorem]{Corollary}

\newtheorem{proposition}[theorem]{Proposition}

\theoremstyle{definition}

\newtheorem{remark}[theorem]{Remark}

\makeatletter
\@namedef{subjclassname@2020}{%
  \textup{2020} Mathematics Subject Classification}
\makeatother

\renewcommand{\O} {\mathcal{O}}

\def\alg{\operatorname{alg}}

\def\aut{\operatorname{Aut}}
\def\lie{\operatorname{Lie}}

\def\AA{\mathbb{A}}
\def\CC{\mathbb{C}}
\def\PP{\mathbb{P}}

\def\dlog{\operatorname{log}_\delta}

\usepackage{verbatim}
\usepackage{color}

\begin{document}

\title{A differential analogue of the wild automorphism conjecture}

\author[Jason Bell]{Jason Bell}
\address{Department of Pure Mathematics \\University of Waterloo \\ Waterloo, ON \\ Canada \\ N2L 3G1}
\email{jpbell@uwaterloo.ca}

\author[Colin Ingalls]{Colin Ingalls}
\address{School of Mathematics and Statistics\\
Carleton University\\
Ottawa, ON \\ K1S 5B6}
\email{cingalls@math.carleton.ca}

\author[Rahim Moosa]{Rahim Moosa}
\address{Department of Pure Mathematics \\University of Waterloo \\ Waterloo, ON \\ Canada \\ N2L 3G1}
\email{rmoosa@uwaterloo.ca}

\author[Matthew Satriano]{Matthew Satriano}
\address{Department of Pure Mathematics \\University of Waterloo \\ Waterloo, ON \\ Canada \\ N2L 3G1}
\email{msatrian@uwaterloo.ca}

\dedicatory{Dedicated to the memory of Zo\'e Chatzidakis (1955--2025)}

\thanks{The authors were supported by Discovery Grants from the National Sciences and Engineering Research Council of Canada.}
\date{\today}


\begin{abstract}
A differential analogue of the conjecture of Reichstein, Rogalski, and Zhang in algebraic dynamics is here established: if $X$ is a projective variety over an algebraically closed field of characteristic zero which admits a global algebraic vector field $v:X\to TX$ such that $(X,v)$ has no proper invariant subvarieties then $X$ is an abelian variety.
Vector fields on abelian varieties with this property are also examined.
Some of the analysis works in the more general context of $D$-varieties over differential fields: projective $D$-varieties without proper $D$-subvarieties are homogeneous.
But the main theorem does not extend: an example of a $D$-variety structure on the projective line without proper $D$-subvarieties is given.
\end{abstract}
\maketitle

\section{Introduction}

\noindent
The theory of difference fields and the theory of differential fields share many similarities, with results in one of the fields often motivating and inspiring analogues in the other. Notable examples of this principle include Picard-Vessiot theory, which was developed in the late 19th century for differential fields (see \cite[Chapt. VIII]{Borel}) with the analogous theory later being developed for difference fields \cite{Fra}; Cantat's theorem \cite{Can} on invariant hypersurfaces, which gives a difference analogue of the Jouanolou-Ghys theorem \cite{Jou,Ghy}; and transcendence results for special values of solutions to differential equations (see \cite[Theorem 5.23]{FN}), which now have satisfying counterparts for solutions to Mahler-type difference equations \cite{AF}.  

In this paper, we consider a differential analogue of a dynamical conjecture due to Reichstein, Rogalski, and Zhang \cite{RRZ}, which asserts that if an automorphism $\sigma$ of a complex projective variety $X$ has the property that $\sigma$ leaves no proper Zariski closed subset invariant, then $X$ is an abelian variety.  Reichstein, Rogalski, and Zhang called such automorphisms {\em wild}, and it is easy to see that abelian varieties always have wild automorphisms coming from suitably chosen translation maps.
A differential analogue of this conjecture can be formulated by observing that for an algebraic vector field $v:X\to TX$, we can view a subvariety $Y$ as {\em invariant} for $(X,v)$ if $v|_Y$ maps $Y$ into $TY$.  Indeed, this point of view is what inspired Cantat to prove his analogue of the Jouanolou-Ghys theorem. In analogy with the terminology used by Reichstein, Rogalski, and Zhang, we say that $v$ is \emph{wild} if there do not exist any proper closed subsets $Y$ of $X$ with the above invariant property.  

In this note we point out that the corresponding differential version of the conjecture of Rogalski, Reichstein, and Zhang holds.

\begin{theorem} \label{thm:main}
Suppose $X$ is a projective variety over an algebraically closed field of characteristic zero.
If $X$ admits a wild vector field then $X$ is an abelian variety.
\end{theorem}

Note that wildness precludes the existence of a nontrivial rational first integral.
Smooth projective varieties $X$ that admit a vector field having no nontrivial rational first integrals were studied in~\cite[$\S$6.4]{abred}, where it was shown that the Albanese map on $X$ is surjective and its generic fibre is birationally equivalent to a homogeneous space for a connected linear algebraic group action.
As it turns out, however, this fact does not play a role in our proof.
Instead, we show directly in~$\S$\ref{sec:Buium}, using~\cite[Thm.~2.1]{Buium} of Buium on $D$-varieties, that under the wildness hypothesis $\aut_0(X)$ will act transitively on $X$, which forces $X=Y\times A$ where $A$ is an abelian variety and $Y$ is a homogeneous space for a connected linear algebraic group.
In particular, $H^1(Y,\mathcal{O}_Y)=0$.
But $Y$ inherits a wild vector field from $X$, and we show in~$\S$\ref{lifting-derivations} that wildness forces $h^1(Y,\mathcal{O}_Y)>0$ unless $Y$ is zero-dimensional.
We thus conclude that $X=A$.
This proof is put together in Section~\ref{Proof}.

Examples of wild vector fields on abelian varieties exist: all nontrivial vector fields on a simple abelian variety are wild.
In general, a vector field on an abelian variety is wild if and only if it admits no nontrivial rational first integrals.
These facts are established in Section~\ref{sec:wvfab}.

Algebraic $D$-varieties make a brief appearance in our proof, and are recalled in Section~\ref{sec:Buium} below.
They can be viewed as vector fields twisted by a derivation on the base field.
One may wonder whether the only projective varieties admitting a $D$-variety structure without proper $D$-subvarieties are abelian varieties -- that is, whether Theorem~\ref{thm:main} holds of $D$-varieties in general.
In a final section, we give an example showing that that is not the case: the projective line admits a ``wild" $D$-variety structure over $(\CC(t)^{\alg},\frac{d}{dt})$.

\begin{remark}
\label{rem:irred}
Irreducible components of algebraic vector fields are invariant -- see, for example~\cite[Thm.~2.1]{kaplansky}.
So, if $X$ admits a wild vector field then it is irreducible.
Moreover, although we do not need it in the sequel, wildness also forces $X$ to be smooth as the singular locus is invariant -- see~\cite[Thm.~12]{Seidenberg}, or~\cite{Hart} for a generalisation.
In fact, we do not even need to assume reduced in Theorem~\ref{thm:main}, the reduced locus is also invariant by~\cite[Lemma 1.8]{kaplansky}.
\end{remark}

\bigskip
 \section{Transitivity of the automorphism group}
 \label{sec:Buium}

\noindent
In this section, we show that a projective variety admitting a wild vector field is a homogeneous space for its automorphism group.
Upon reading an earlier version of this paper, Serge Cantat communicated a simpler proof of this fact to us, using the holomorphic flow induced by an algebraic vector field.
We present our original proof here, which is based on results of Buium appearing in~\cite{Buium}.
One advantage of our approach is that it works in the more general category of algebraic $D$-varieties, and this may be of independent interest.

We begin by briefly recalling the setting of $D$-algebraic geometry.
Fix an algebraically closed differential field $(K,\delta)$ of characteristic zero with (algebraically closed) field of constants~$k$.
We will need the notion of a {\em $D$-variety over $(K,\delta)$} in the sense of Buium~\cite{Buium}: an algebraic variety $Y$ over $K$ equipped with a derivation on $\mathcal{O}_Y$ that extends $\delta$ on $K$.
Such a derivation induces, and indeed is determined by, a regular section to the {\em prolongation} $\tau Y\to Y$, a certain torsor of the tangent space $TY\to Y$ induced by $\delta$.
In affine co-ordinates $x=(x_1,\dots,x_n)$ and $y=(y_1,\dots,y_n)$, the defining equations for $\tau Y$ are of the form
$$f^\delta(x)+\sum_{i=1}^n\frac{\partial f}{\partial x_i}(x)y_i=0$$
as $f\in K[x]$ ranges among a generating set for the ideal of (the corresponding affine chart of) $Y$, and $f^\delta$ is the polynomial obtained from $f$ by applying $\delta$ to the coefficients.
Given a $D$-variety $(Y,s)$, a $D$-\emph{subvariety} is a subvariety $Y'\subseteq Y$  such that $s(Y')\subseteq\tau(Y')$.
For a more detailed (but still brief) introduction to these notions, including additional references, see~\cite[$\S$2.1]{BLM}.

If $K=k$, so that $\delta$ is the trivial derivation, then a $D$-variety is just a variety equipped with a vector field, and the $D$-subvarieties are just the invariant subvarieties of that vector field.
So the context of $D$-varieties extends that of algebraic vector fields.
We will be interested in a situation where it may be that $K\neq k$ but where $Y$ is defined over~$k$ in the sense that it is the base extension of a variety $X$ over $k$.
We denote this by $Y=X_K$.
Note that in this case $\tau Y=(TX)_K$.


\begin{proposition}\label{prop:aut0-wild}
Suppose $(K,\delta)$ is an algebraically closed differential field of characteristic zero with field of constants~$k$.
Suppose $X$ is an irreducible projective variety over $k$ and $Z\subseteq X$ is a closed subvariety that is $\aut_0(X)$-invariant.
Then $Z_K$ is a $D$-subvariety of $(X_K,s)$, for any $D$-variety structure $s$ on $X_K$ over $(K,\delta)$.

In particular, every $\aut_0(X)$-invariant subvariety of~$X$ is invariant for any algebraic vector field on~$X$.
\end{proposition}

\begin{proof}
The ``in particular" clause is precisely the specialisation to when $K=k$.

Let $(F,\delta)$ be a differentially closed field extending $(K,\delta)$ whose field of constants is~$k$.
That it is {\em differentially closed} means that every finite system of differential-polynomial equations and inequations over~$F$ that has a solution in some differential field extension of $(F,\delta)$ already has a solution in $(F,\delta)$.
We view $(X_F,s_F)$ as a $D$-variety over $(F,\delta)$.
To show that $Z_K$ is an invariant subvariety of $(X_K,s)$ it suffices to show that $Z_F$ is a $D$-subvariety of $(X_F,s_F)$.

By a theorem of Buium~\cite[Thm~II.2.1]{Buium}, because $X$ is projective and $(F,\delta)$ is differentially closed, $(X_F,s_F)$ is isomorphic to the trivial $D$-variety $(X_F,0)$.
In fact, there is $\sigma \in \aut_0(X)(F)$ such that the following diagram commutes:
$$\xymatrix{
(TX)_F\ar[r]^{\tau(\sigma)}&(TX)_F\\
X_F\ar[u]^{0}\ar[r]^{\sigma}&X_F\ar[u]^{s_F}
}$$
Let us sketch how~$\sigma$ is found.
As $X$ is defined over~$k$, $\tau(X_K)=T(X_K)$, and so we can also view $s$ as an algebraic vector field on $X_K$, which we will denote by~$v_s$.
But algebraic vector fields on $X_K$ can be identified with $K$-points of $\lie(\aut_0(X))$, see~\cite[II.2.5]{Buium}.
So $v_s\in \lie(\aut_0(X))(K)$.
Now, for any connected algebraic group~$G$ over the constants there is a {\em logarithmic-derivative on~$G$} which is a certain $\delta$-rational surjective crossed homomorphism
$\dlog:G(F)\to \lie(G)(F)$ whose fibre above~$0$ is $G(k)$.
(It is in order to have surjectivity that we required $(F,\delta)$ to be differentially closed.)
Let $\sigma\in \aut_0(X)(F)$ be such that $\dlog(\sigma^{-1})=v_s$.
This choice of $\sigma$ will make the above diagram commute.
See~\cite[II.2.10]{Buium} for a detailed argument.

In any case, we have $D$-varieties $(X_F,0)$ and $(X_F,s_F)$ over $(F,\delta)$, with~$\sigma$ an isomorphism between them.
Note that $Z_F$ is a $D$-subvariety of $(X_F,0)$ because~$Z$ is also defined over~$k$ and hence $\tau(Z_F)=(TZ)_F$.
It follows that $\sigma(Z_F)$ is a $D$-subvariety of $(X_F,s_F)$.
But $\sigma(Z_F)=Z_F$ by $\aut_0(X)$-invariance.
So $Z_F$ is a $D$-subvariety of $(X_F,s_F)$ and hence $Z_K$ is a $D$-subvariety of $(X_K,s)$.
\end{proof}

\begin{corollary} \label{cor:Rahim}
Suppose $(Y,s)$ is a projective $D$-variety over $(K,\delta)$ that has no proper $D$-subvarieties.
Then $\aut_0(Y)$ acts transitively on~$Y$.

In particular, a projective variety admitting a wild algebraic vector field is homogeneous for an algebraic group action.
\end{corollary}

\begin{proof}
Again, the ``in particular" clause is the specialisation to when $\delta=0$.

By Remark~\ref{rem:irred}, or rather its $D$-variety analogue, the absence of $D$-subvarieties does imply that~$Y$ is irreducible.

By~\cite[Theorem~II.1.1]{Buium} the only way a projective variety~$Y$ admits a $D$-variety structure over $(K,\delta)$ is if~$Y$ is defined over the constants of $(K,\delta)$, namely~$k$.
So we may assume that $Y=X_K$ for some projective variety~$X$ over~$k$, and we show that $\aut_0(X)$ acts transitively on~$X$.
Indeed, let $Z$ be the Zariski closure of the orbit of any $k$-point of~$X$.
Then $Z$ is $\aut_0(X)$-invariant, and hence $Z_K$ is a $D$-subvariety of $(X_K,s)$ by Proposition~\ref{prop:aut0-wild}.
By assumption, this forces $Z_K=X_K$, and hence $Z=X$.
So every $k$-orbit of $\aut_0(X)$ is Zariski dense in~$X$.
But this is only possible if the action is transitive.
\end{proof}

\bigskip
\section{Lifting derivations}
\label{lifting-derivations}

\noindent
Fix an algebraically closed field $k$ of characteristic zero.
In this section, we show that if $X$ is a smooth projective variety over $k$ with trivial Albanese, then every algebraic vector field on $X$ can be lifted to a $k$-linear derivation of a homogeneous coordinate ring of $X$.

Given $U\subset X$ an open subset, a vector field $v \in H^0(U,TU)$, and a line bundle~$\mathcal{L}$ on $U$, we define a \emph{lift} $\widetilde{{v}}$ of ${v}$ to $\mathcal{L}$ as a $k$-linear map
  $ \widetilde{{v}}: \mathcal{L} \to \mathcal{L}$
  satisfying
 $$\widetilde{{v}}(fs) = {v}(f)s+f\widetilde{{v}}(s),$$
for local sections $f$ and $s$ of $\O_U$ and $\mathcal{L}$ respectively.
Here we are identifying $v$ with the derivation on $\O_X$ it induces.
Notice that if $v\in H^0(X,TX)$ is a global section, then the local lifts $\widetilde{\mathcal{V}}_v$ of $v$ form a subsheaf of $\mathcal{H}om_k(\mathcal{L},\mathcal{L})$.

A more general version of this result is given in~\cite{CL77}.  We include a proof of our case of interest for completeness. 

  \begin{proposition}\label{prop:OX-torsor-lifts} Let $X$ be a smooth variety over $k$, $\mathcal{L}$ a line bundle on $X$, and $v \in H^0(X,TX)$. 
     Then 
$\widetilde{\mathcal{V}}_v$ is an $\O_X$-torsor. Hence, the obstruction to the existence of a global lift lives in $H^1(X,\O_X)$.
     \end{proposition}
  \begin{proof}
First, there is a natural action of $\O_X$ on $\widetilde{\mathcal{V}}_v$ given as follows. If $f\in\O_X(U)\simeq\mathrm{Hom}(\mathcal{L}|_U,\mathcal{L}|_U)$ and $\widetilde{{v}}\colon\mathcal{L}|_U\to\mathcal{L}|_U$ is a lift of $v|_U$, then $f+\widetilde{{v}}$ is also a lift of $v|_U$.

Next, lifts of $v$ exist locally:~over any open subset $U\subset X$ where $\mathcal{L}|_U\simeq\O_U$, we may choose the lift of $v|_U$ to be $v|_U$ itself. Finally, if $V\subset X$ is any open and $\widetilde{{v}}_1$ and $\widetilde{{v}}_2$ are two lifts of $v|_V$, then $(\widetilde{{v}}_1-\widetilde{{v}}_2)(fs)=f\cdot(\widetilde{{v}}_1-\widetilde{{v}}_2)(s)$ where $f$ is any section of $\O_V$ and $s$ is any section of $\mathcal{L}|_V$. Therefore, 
$$\widetilde{{v}}_1-\widetilde{{v}}_2\in \mathcal{H}om_{\O_X}(\mathcal{L}|_V,\mathcal{L}|_V)=\O_V,$$
which proves that the lifts of $v$ form an $\O_X$-torsor.

To finish the proof, we recall that $\O_X$-torsors are classified by $H^1(X,\O_X)$, so the class $[\widetilde{\mathcal{V}}_v]\in H^1(X,\O_X)$ vanishes if and only if $\widetilde{\mathcal{V}}_v$ is the trivial torsor, which by definition, means $\widetilde{\mathcal{V}}_v$ has a global section. So $v\in H^0(X,TX)$ has a global lift if and only if $[\widetilde{\mathcal{V}}_v]\in H^1(X,\O_X)$ vanishes.
\end{proof}

The following result is also proved in~\cite[Thm.~2]{CL73} over $\mathbb{C}$ using analytic techniques.  We present a very different algebraic proof.  

\begin{corollary}
Let $X$ be a smooth projective variety over $k$ of dimension at least one, and let $v\in H^0(X,TX)$.
If $H^1(X,\mathcal{O}_X)=0$ then $v$ is not wild. \label{cor:notwild}
\end{corollary}

\begin{proof}
Let $\mathcal{L}$ be a very ample invertible sheaf. 
By Proposition \ref{prop:OX-torsor-lifts}, we can lift ${v}$ to $\mathcal{L}$.
Such a lift then defines a degree-preserving graded derivation $\delta$ of the homogeneous coordinate ring $$R:=\bigoplus_{n\ge 0} H^0(X,\mathcal{L}^{\otimes n}).$$
Note that $X=\operatorname{Proj}(R)$.
Let $f\in H^0(X,\mathcal{L})\setminus \{0\}$ be an eigenvector for $\delta$.  Then the homogeneous ideal $fR$ defines a $\delta$-invariant ideal of $R$.
We claim that $fR$ gives rise to an invariant codimension one subvariety of $(X,v)$.
Indeed, since $\mathcal{L}$ is very ample, $R$ is generated in degree one, and, since $R$ has Krull dimension at least two, Krull's principal ideal theorem implies that there is some $g\in H^0(X,\mathcal{L})$ that is not in the radical of $fR$. 
The degree zero piece of $R[1/g]$ can be identified with $\mathcal{O}(U)$ for some dense Zariski open $U\subseteq X$, and $\delta$ induces a derivation of $R[1/g]$ that agrees with the derivation given by $v$.
That $f$ was an eigenvector for $\delta$ in $R$ ensures that $I:=fg^{-1}\mathcal{O}(U)$ is a $\delta$-invariant ideal.
As $\O(U)$ extends $\mathbb Q$, the radical of $I$ is also $\delta$-invariant (see, for example,~\cite[Lemma~1.8]{kaplansky}).
It therefore defines an invariant subvariety of $(U,v)$, whose Zariski closure in $X$ will be an invariant subvariety $Z$ of $(X,v)$.
It remains only to verify that $Z$ is nonempty, but this follows from the fact that $I$ is proper: else there is some $n\ge 1$ and some $h\in H^0(X,\mathcal{L}^{\otimes n})$ such that $(fg^{-1})(hg^{-n}) =1$, which then gives that $fh=g^{n+1}$ implying that $g$ is in the radical of $fR$, a contradiction.
\end{proof}

\bigskip
\section{Proof of Theorem \ref{thm:main}}
\label{Proof}

\noindent
We can now give a short proof of our main result.
Suppose $X$ is a projective variety over an algebraically closed field $k$ and $v$ is a wild algebraic vector field on $X$.
By Corollary \ref{cor:Rahim}, ${\rm Aut}_0(X)$ acts transitively on $X$.
Then by \cite[Theorem 5.2]{San} we can identify $X$ with $Y\times A$, where $A$ is the Albanese of $X$ and $Y=H/P$ with $H$ a connected affine algebraic group and $P$ a parabolic subgroup, and ${\rm Aut}_0(X) \cong {\rm Aut}_0(Y) \times A$.
In particular, $H^1(Y,\O_Y)=0$.
Since $H^0(X, TX)={\rm Lie}({\rm Aut}_0(X))$, it follows that $v$ can be written uniquely as $v_1+v_2$ where we may identify $v_1$ with a vector field $v_1'$ on $Y$ and $v_2$ with a vector field $v_2'$ on $A$. 
If $Y$ has dimension at least one, then, by Corollary \ref{cor:notwild}, there is an invariant proper subvariety $Z$ for $(Y,v_1')$ and by construction $Z\times A$ is invariant for both $(X,v_1)$ and $(X,v_2)$ and hence is invariant for $(X,v)$, thus contradicting wildness.
It follows that $Y$ is a point and so $X$ is abelian, as desired.  
 \qed

\bigskip
\section{Wild vector fields on abelian varieties}
\label{sec:wvfab}

\noindent
We first point out that on simple abelian varieties every nontrivial algebraic vector field is wild.
Fix an algebraically closed field $k$ of characteristic zero.

\begin{proposition}
\label{prop:saw}
Suppose $A$ is a simple abelian variety over $k$.
Every nontrivial algebraic vector field on $A$ is wild.
\end{proposition}

\begin{proof}
Let $(F,\delta)$ be a differentially closed field with field of constants $k$.
Buium's theorem~\cite[Thm~2.1]{Buium} gives us $\sigma\in\aut_0(A)(F)$ such that $\sigma:(A_F,v_F)\to(A_F,0)$ is an isomorphism of $D$-varieties.
But as $\aut_0(A)=A$, we have that $\sigma$ is just translation by some $P\in A(F)$.
Suppose $X\subseteq A$ is an invariant subvariety for $v$.
Then $\sigma(X_F)=P+X_F$ is a $D$-subvariety of the trivial $D$-variety $(A_F,0)$.
But every $D$-subvariety of $(A_F,0)$ is defined over the constant field~$k$.
Since $X_F$ is also defined over $k$, we have that for any field automorphism $\alpha\in\aut(F/k)$, $\alpha(P)-P\in\operatorname{Stab}(X_F)$.
If $\operatorname{Stab}(X_F)$ is finite then this tells us that $P$ has a finite orbit under $\aut(F/k)$, implying that $P\in A(k)$.
It would follow that translation by $P$ is an isomorphism between the vector fields $(A, v)$ and $(A,0)$ over $k$, which is only possible if $v=0$ already.
Hence, $\operatorname{Stab}(X_F)$ is a positive-dimensional algebraic subgroup of $A$.
Now, by simplicity, this would mean that $\operatorname{Stab}(X_F)=A_F$,
which in turn implies that $X=A$.
\end{proof}

Regarding general abelian varieties, we have the following classification of wild vector fields:

\begin{proposition}
\label{prop:aw}
Suppose $v$ is an algebraic vector field on an abelian variety $A$ over $k$.
Then $v$ is wild if and only if it admits no nontrivial rational first integrals.
\end{proposition}

\begin{proof}
Suppose $f$ is a nontrivial rational first integral of $(A,v)$.
Let $\delta_v$ denote the derivation on $k(A)$ induced by $v$, this means that $f\in k(A)\setminus k$ with $\delta_v(f)=0$.
In that case, any level set of $f$ over $k$ is an invariant subvariety of $(A,v)$, and by nontriviality some level set will be proper.
Hence $(A,v)$ is not wild.

For the converse we use the truth of the differential Dixmier-Moeglin equivalence for {\em isotrivial} algebraic vector fields, as studied in~\cite{BLM}.
That $(A,v)$ is isotrivial is what Buium's theorem says.
That $(A,v)$ admits no nontrivial rational first integrals means that it is {\em $\delta$-rational}, and hence by~\cite[Fact~2.9]{BLM} it is {\em $\delta$-locally closed}.\footnote{For the model theorist, this is simply the fact that as the generic type of $(A,v)$ is internal to the constants, being weakly orthogonal to the constants implies that the type is isolated.}
That is, there exists a maximum proper $D$-subvariety $X\subset A$.
But as $v$ is translation-invariant on $A$ (all algebraic vector fields on an abelian variety are translation-invariant), every translate of $X$ by a $k$-point of $A$ will also be a proper $D$-subvariety, and hence will be contained in $X$ by maximality.
The only way this is possible is if $X$ is empty -- which says precisely that $(A,v)$ is wild.
\end{proof}

The above proposition is the differential analogue of~\cite[Theorem~1.7]{bellghioca22} which treats algebraic dynamics of dynamical degree one self-maps on semiabelian varieties.

\bigskip
\section{A wild $D$-variety structure on the projective line}

\noindent
Half of our proof of Theorem~\ref{thm:main}, namely the part appearing in~$\S$\ref{sec:Buium} that allowed us to deduce homogeneity from the existence of a wild vector field, worked in the more general setting of $D$-varieties.
It is therefore natural to ask about the Theorem itself in that setting:
{\em If a projective variety~$X$ over a constant field~$k$ admits a $D$-variety structure over a differential field extension $(K,\delta)$, with no proper $D$-subvarieties, must~$X$ be an abelian variety?}
The answer is ``no" as the following example illustrates.

Let $\PP=\PP^1_\CC$ be the complex projective line.
Let $K=\CC(t)^{\alg}$ be the algebraic closure of the complex function field in one variable, equipped with the $\CC$-derivation $\delta:=\frac{d}{dt}$.
We will describe a $D$-variety structure on $\PP_K$ over $(K,\delta)$ that admits no proper $D$-subvarieties.

Since $\tau\PP_K=T\PP_K$, a $D$-variety structure is given by a vector field on $\PP_K$.
Letting $\AA=\AA^1_\CC$, the transition map $\phi:\AA\setminus\{0\}\to\AA\setminus\{0\}$ that defines $\PP$ as a glueing of two copies of the affine line is given by $\phi(x)=\frac{1}{x}$.
So $d\phi:T(\AA\setminus\{0\})\to T(\AA\setminus\{0\})$, given by
$(x,v)\mapsto (\frac{1}{x},\frac{-v}{x^2})$, gives $T\PP$ as a glueing of two copies of $T\AA$.
To define an algebraic vector field on $\PP_K$ we need two affine vector fields, $s_0,s_1:\AA_K\to T\AA_K$, such that $d\phi\circ s_0=s_1\circ\phi$ on $\AA_K\setminus\{0\}$.
To that end, let $s_0(x):=(x,x^2-t)$ and $s_1(x):=(x,tx^2-1)$.
Then
\begin{eqnarray*}
d\phi(s_0(x))
&=&
d\phi(x,x^2-t)\\
&=&
(\frac{1}{x},\frac{-x^2+t}{x^2})\\
&=&
(\frac{1}{x},t(\frac{1}{x^2})-1)\\
&=&
s_1(\frac{1}{x})\\
&=&
s_1(\phi(x)).
\end{eqnarray*}
So $s_0$ and $s_1$ patch to give a vector field $s:\PP_K\to T\PP_K$, and we have a $D$-variety $(\PP_K,s)$ over $(K,\delta)$.
Finally, we claim that $(\PP_K,s)$ has no proper $D$-subvarieties over~$K$.
As every irreducible component of a $D$-subvariety is again a $D$-subvariety, we need only show that $(\PP_K,s)$ has no $D$-points in $(K,\delta)$.
Reading through the charts, this means that, for $i=0,1$, there is no $x\in K$ such that
$s_i(x)=(x,\delta(x))$.
That is, we need to show that neither
\begin{equation}
\label{s0}
\frac{dx}{dt} =x^2-t
\end{equation}
nor
\begin{equation}
\label{s1}
\frac{dx}{dt} = tx^2-1
\end{equation}
has any solutions in $\CC(t)^{\alg}$.
Because these patch by the change of variables $u=\frac{1}{x}$, and~$0$ is not a solution to~(\ref{s1}), it suffices to show that the Ricatti equation~(\ref{s0}) has no algebraic solutions.
This is likely well known, but here is an argument using the algebraic transcendence of solutions to the Airy equation:
Suppose~$x$ is a solution to~(\ref{s0}) in some differentially closed field extending $(\CC(t)^{\alg},\frac{d}{dt})$.
Let $u$ be a solution to $u'=-xu$.
Then $u''=x^2u-x'u=x^2u-(x^2-t)u=tu$.
That is, $u$ is a solution to the Airy equation, from which it follows that $u$ and $u'$ are algebraically independent over $\CC(t)$, see~\cite[Theorem~1]{bornemann} for example.
Hence $u'=-xu$ implies that~$x\notin\CC(t)^{\alg}$, as desired.

\bigskip
\section*{Acknowledgment} The authors are grateful to BIRS, which hosted the workshop ``Noncommutative Geometry and Noncommutative Invariant Theory'' (22w5084), during which part of this project was completed.  


\end{document}